\documentclass[11pt,reqno]{amsart}
\usepackage[none]{hyphenat}
\usepackage[T1]{fontenc}
\usepackage[utf8]{inputenc}
\usepackage{lmodern}
\usepackage{amsmath,amssymb,amsthm,mathtools}
\usepackage{microtype}
\usepackage{enumitem}
\usepackage{hyperref}

\hypersetup{
  colorlinks=true,
  linkcolor=blue,
  citecolor=blue,
  urlcolor=blue
}

\numberwithin{equation}{section}
\newcommand{\doi}[1]{\href{https://doi.org/#1}{\nolinkurl{#1}}}
\newcommand{\arxiv}[1]{\href{https://arxiv.org/abs/#1}{arXiv:\nolinkurl{#1}}}

\newtheorem{theorem}{Theorem}[section]
\newtheorem{proposition}[theorem]{Proposition}
\newtheorem{lemma}[theorem]{Lemma}
\newtheorem{corollary}[theorem]{Corollary}

\theoremstyle{definition}

\newtheorem{example}[theorem]{Example}

\theoremstyle{remark}
\newtheorem{remark}[theorem]{Remark}

\newcommand{\Ric}{\operatorname{Ric}}

\newcommand{\C}{\mathbb C}
\newcommand{\R}{\mathbb R}

\newcommand{\OO}{\mathcal O}

\newcommand{\Vol}{\operatorname{Vol}}
\newcommand{\tr}{\operatorname{tr}}
\newcommand{\Tr}{\operatorname{Tr}}

\newcommand{\supp}{\operatorname{supp}}

\newcommand{\GL}{\operatorname{GL}}
\newcommand{\Frac}{\operatorname{Frac}}
\newcommand{\trdeg}{\operatorname{trdeg}}

\newcommand{\lap}{\Delta_{\R}}
\newcommand{\lapc}{\Delta_{\C}}

\newcommand{\Pn}{\mathcal P_N}

\newcommand{\ip}[2]{\left\langle#1,#2\right\rangle}
\title[ Polynomial Growth Holomorphic Functions]
{A Dimension Estimate for Polynomial Growth Holomorphic Functions}

\author{Haohao Wang}


\begin{document}

\begin{abstract}
On a complete noncompact K\"ahler manifold  $M^n$ of complex dimension $n$ with nonnegative Ricci curvature, we prove that polynomial
growth holomorphic functions of degree $d$ ($d\geq 0$) have an dimension upper bound $C(n)(1+d)^n$. In particular, $\dim_{\C}\OO_d(M)\leq C'(n)d^n,$ for $d\geq1$,  and the estimate  is sharp.
\end{abstract}

\maketitle

\section{Introduction}

Let $(M^m,g,p)$ be a complete connected noncompact Riemannian manifold. Write $r(x)=d_g(p,x)$. For $d\geq0$, let
\begin{equation}\label{eq:harmonic-space}
 {\mathcal H}_d(M)=\{u\in C^\infty(M,\mathbb R):\Delta_{\mathbb R} u=0,\quad
 |u(x)|\leq C_u(1+r(x))^d\text{ for all }x\in M\}.
\end{equation}
Colding--Minicozzi \cite{CM97} proved the finite dimensionality of this space when $\Ric_g\geq0$. Li \cite{Li97} gave a different proof. The sharp power estimate in the Riemannian setting is
\begin{equation}\label{eq:harmonic-bound}
 \dim_{\mathbb R}{\mathcal H}_d(M)\leq C(m)d^{m-1},\qquad d\geq1;
\end{equation}
see \cite{CM98,Li97}.
 
Now assume that $(M^n,g,J,\omega)$ is K\"ahler, where $n$ is the complex dimension. We use
\begin{equation}\label{eq:holomorphic-space}
 \OO_d(M)=\{f\in\OO(M): |f(x)|\leq C_f(1+r(x))^d
 \text{ for all }x\in M\}.
\end{equation}
The constant $C_f$ may depend on $f$. The space in \eqref{eq:holomorphic-space} does not depend on the choice of $p$, and all dimensions of spaces of holomorphic functions are taken over $\C$. Applying \eqref{eq:harmonic-bound} to real and imaginary parts gives a bound with power $2n-1$. This does not recover the power $n$ of the complex Euclidean model when $n>1$.

Under nonnegative holomorphic bisectional curvature and Euclidean volume growth, Ni \cite{Ni04} proved the precise comparison
$\dim_{\C}\OO_d(M)\leq\dim_{\C}\OO_d(\C^n)$. Chen--Fu--Yin--Zhu \cite{CFYZ06} removed the volume-growth assumption. Liu \cite{Liu16} obtained the same comparison under nonnegative holomorphic sectional curvature, using a three-circle theorem. In \cite[Question~1]{Liu21}, Liu asked whether the power estimate with a constant depending only on $n$ holds under nonnegative Ricci curvature alone. He proved
\begin{equation}\label{eq:liu-result}
 \dim_{\C}\OO_d(M)\leq C(n,\alpha)d^n
\end{equation}
assuming $\Vol B(p,R)\geq\alpha R^{2n}$ for all $R>0$. His argument uses the complex structure of tangent cones and a determinant form of logarithmic convexity on those cones.

The main result considered in this paper is the following.

\begin{theorem}\label{thm:main}
Let $M^n$ be a complete connected noncompact K\"ahler manifold with $\Ric_g\geq0$. There is a constant $C(n)>0$ such that
\begin{equation}\label{eq:main}
 \dim_{\C}\OO_d(M)\leq C(n)(1+d)^n
\end{equation}
for every real $d\geq0$. In particular,
\begin{equation}\label{eq:main-large}
 \dim_{\C}\OO_d(M)\leq C'(n)d^n,\qquad d\geq1.
\end{equation}

\end{theorem}

The power in \eqref{eq:main-large} is sharp, since
\begin{equation}\label{eq:euclidean-dimension}
 \dim_{\C}\OO_d(\C^n)=\binom{n+\lfloor d\rfloor}{n}.
\end{equation}
The assertion is a sharp-power estimate, not the precise Euclidean comparison in \cite{Ni04,CFYZ06,Liu16}. The factor $1+d$ in \eqref{eq:main} is necessary for a statement valid down to $d=0$, because constants belong to $\OO_d(M)$ for every $d\geq0$.

We briefly describe the argument. Fix a finite-dimensional subspace $V\subset\OO_d(M)$ containing the constants, and let $F=(f_1,\ldots,f_N)^t$ be a basis with $f_1=1$. If $p_t(p,x)$ is the heat kernel of the real Laplacian, put $d\nu_t=p_t(p,x)dV_g(x)$ and consider
\begin{equation}\label{eq:intro-functional}
 \beta_F(t)=\min_{\substack{A=A^*>0\\ \det A=1}}
 \int_M\log(F^*AF)\,d\nu_t.
\end{equation}
The minimizing form makes the projective evaluations balanced. This is a measure-dependent version of the projective balancing construction of Bourguignon--Li--Yau \cite{BLY94} and Donaldson \cite[Section~2.2]{Donaldson09}. We prove the existence and continuity needed here directly; in particular, no uniform bound for the condition numbers of the minimizing matrices is used.

For the minimizing matrix $A_t$, set $g=A_t^{1/2}F$ and $S=|g|^2$. Give the trivial holomorphic line bundle the metric $h=S^{-1}$. The resulting holomorphic sections satisfy
\begin{equation}\label{eq:intro-energy}
 \sum_{a=1}^N|g_a|_h^2=1,\qquad
 \sum_{a=1}^N|D^h g_a|_h^2=\frac12\lap\log S.
\end{equation}
Here $D^h$ is the Chern connection, and the norm of $D^h g_a$ is the full real covariant-derivative norm. A heat-trace argument for this single line bundle gives
\begin{equation}\label{eq:intro-slope}
 N\leq C(n)\bigl(1+t\beta_F'(t)\bigr)^n.
\end{equation}
On the other hand, polynomial growth implies $\beta_F(t)\leq d\log t+O_{V,d}(1)$ for $t\geq1$. Integrating the lower bound for $\beta_F'$ obtained from \eqref{eq:intro-slope} proves \eqref{eq:main}.

The normalized spectral viewpoint is also related to the approach in the manuscript \cite{Wang}. Its spectral almost-rigidity results are not used in the proof below. Related uses of holomorphic sections and connection spectra on compact K\"ahler manifolds appear in \cite{Yuan26}. Here the metric on the line bundle is chosen from the function space itself, and the counting argument takes place on bounded domains in the original smooth manifold. The analytic inputs are heat-kernel estimates under a Ricci lower bound and domination of the heat semigroup of a unitary connection.

The paper is organized as follows. Section~2 records heat-kernel and connection estimates. Section~3 proves a volume-independent count for balanced sections. Section~4 constructs the minimizing logarithmic functional and proves its differentiation formula. Section~5 gives the Chern-connection identity and the proof of Theorem~\ref{thm:main}. Section~6 discusses global Bergman kernels, algebraic dependence, examples, and the distinction from local three-circle estimates.
\section{preliminaries}

The real Laplacian is
\[
 \lap=\operatorname{div}_g\nabla,
\]
so that $-\lap$ is nonnegative on compactly supported functions. The heat semigroup is $P_t=e^{t\lap}$. On a K\"ahler manifold, if $Z_1,\ldots,Z_n$ is a unitary frame of $T^{1,0}M$, we write
\begin{equation}\label{eq:lap-convention}
 \lapc u=\sum_{i=1}^n(\nabla du)(Z_i,\overline Z_i),
 \qquad \lap=2\lapc,\qquad dV_g=\frac{\omega^n}{n!}.
\end{equation}
For a real function $u$, $|du|^2=2|\partial u|^2$. For a complex-valued function, $|du|^2$ is the sum of the energies of its real and imaginary parts. Hermitian inner products are conjugate linear in the first argument. A superscript $*$ denotes the conjugate transpose of a matrix.

In Sections~2 and 3, $m$ denotes real dimension. Let $M^m$ be complete and noncompact with $\Ric_g\geq0$. It is stochastically complete; hence its minimal heat kernel satisfies
\begin{equation}\label{eq:probability}
 p_t(x,y)>0,\qquad \int_Mp_t(x,y)\,dV_g(y)=1.
\end{equation}
Fix $p\in M$ and define
\begin{equation}\label{eq:heat-measure}
 d\nu_t(x)=p_t(p,x)dV_g(x),\qquad
 V_p(R)=\Vol B(p,R).
\end{equation}
The standard heat-kernel estimates and stochastic completeness used here can be found in \cite{LiYau86,Grigoryan09}.

\begin{proposition}\label{prop:heat}
There are positive constants depending only on $m$ such that
\begin{equation}\label{eq:gaussian}
 \frac{c_m}{V_p(\sqrt t)}e^{-C_m r(x)^2/t}
 \leq p_t(p,x)\leq
 \frac{C_m}{V_p(\sqrt t)}e^{-c_m r(x)^2/t}
\end{equation}
for every $t>0$ and $x\in M$. Consequently, for each $b>0$,
\begin{equation}\label{eq:local-density}
 \frac{c(m,b)}{V_p(\sqrt t)}
 \leq p_t(p,x)\leq\frac{C(m,b)}{V_p(\sqrt t)}
 \quad\text{on }B(p,b\sqrt t).
\end{equation}
For every $q\geq0$,
\begin{equation}\label{eq:moment}
 \int_M r(x)^q\,d\nu_t(x)\leq C(m,q)t^{q/2}.
\end{equation}
There is $a=a(m)>1$ such that
\begin{equation}\label{eq:tail}
 \nu_t\bigl(M\setminus B(p,a\sqrt t)\bigr)\leq\frac18
 \quad\text{for all }t>0.
\end{equation}
\end{proposition}

\begin{proof}
The two-sided bounds in \eqref{eq:gaussian} follow from the Li--Yau estimates and relative volume comparison; a polynomial factor arising when changing the volume center can be absorbed into the Gaussian. Equation~\eqref{eq:local-density} is immediate. To check the moment estimate, put $R=\sqrt t$ and decompose the complement of $B(p,R)$ into the annuli
\[
 B(p,2^{j+1}R)\setminus B(p,2^jR),\qquad j\geq0.
\]
Bishop--Gromov comparison gives $V_p(2^{j+1}R)\leq 2^{(j+1)m}V_p(R)$. The upper bound in \eqref{eq:gaussian} therefore yields
\[
 \int_Mr^q\,d\nu_t
 \leq R^q+C_mR^q\sum_{j\geq0}
 2^{(j+1)(m+q)}e^{-c_m4^j}.
\]
The series converges. Taking $q=2$ and applying Markov's inequality proves \eqref{eq:tail} for a sufficiently large $a(m)$.
\end{proof}

The estimates above imply Gaussian integrability of every function of polynomial growth. They also imply that $t\mapsto\nu_t$ is continuous in total variation for $t>0$. Indeed, heat-kernel continuity and \eqref{eq:probability} give this by Scheff\'e's lemma. The same conclusion follows by dominated convergence on compact sets and the uniform Gaussian tails on compact time intervals.

We will only differentiate heat averages of smooth functions whose first derivatives and Laplacians have polynomial growth. The following elementary form of the semigroup identity is convenient.

\begin{lemma}\label{lem:heat-derivative}
Suppose that $u\in C^\infty(M,\R)$ and, for some $C,q<\infty$,
\begin{equation}\label{eq:polynomial-test}
 |u|+|du|+|\lap u|\leq C(1+r)^q.
\end{equation}
Then $t\mapsto P_tu(p)$ is continuously differentiable for $t>0$, and
\begin{equation}\label{eq:heat-derivative}
 \frac{d}{dt}P_tu(p)=P_t(\lap u)(p).
\end{equation}
If $u=u_\theta$ depends continuously on a finite-dimensional parameter, together with $\lap u_\theta$, and \eqref{eq:polynomial-test} holds uniformly on compact parameter sets, then the right-hand side is jointly continuous in $(\theta,t)$.
\end{lemma}

\begin{proof}
For completeness, the usual localization argument can be written using the diffusion $X_s$ with generator $\lap$, started at $p$; this is Brownian motion with twice the usual $\frac12\lap$ time normalization. We use the stopped It\^o formula; see \cite[Chapters~2 and 3]{Hsu02}. If $\tau_R$ is the first exit time from $B(p,R)$, then
\[
 u(X_{t\wedge\tau_R})=u(p)+M_{t\wedge\tau_R}
 +\int_0^{t\wedge\tau_R}\lap u(X_s)\,ds,
\]
where the quadratic variation of the martingale is
$2\int_0^{t\wedge\tau_R}|du|^2(X_s)\,ds$.
By \eqref{eq:moment} and \eqref{eq:polynomial-test},
\[
 \mathbb E\int_0^t\bigl(|du|^2+|\lap u|\bigr)(X_s)\,ds<\infty.
\]
Stochastic completeness gives $\tau_R\to\infty$ almost surely. The martingales converge in $L^2$, and the drift integrals converge in $L^1$. Passing to the limit and taking expectations gives
\begin{equation}\label{eq:duhamel}
 P_tu(p)-u(p)=\int_0^tP_s(\lap u)(p)\,ds.
\end{equation}
For $s$ in a compact subinterval of $(0,\infty)$, the integrand is continuous by the Gaussian bound and polynomial growth. This proves \eqref{eq:heat-derivative}. The parameter statement follows from the same domination. No assertion concerning preservation of plurisubharmonicity is involved.
\end{proof}

\subsection{Heat-trace bounds for a unitary connection}

Let $(L,h)$ be a Hermitian line bundle over $M$, and let $D$ be a smooth unitary connection. For a smooth relatively compact domain $\Omega\subset M$, let
\[
 \mathcal L_{D,\Omega}=D^*D
\]
be the Dirichlet realization on $L^2(\Omega,L;dV_g)$. Equivalently, it is the nonnegative self-adjoint operator associated with the closure of the quadratic form
\[
 q_D(s)=\int_\Omega|Ds|_h^2\,dV_g,
 \qquad s\in C_c^\infty(\Omega,L).
\]
All eigenvalue multiplicities of $\mathcal L_{D,\Omega}$ are complex multiplicities.

\begin{lemma}\label{lem:domination}
Let $\Delta_{\R,\Omega}$ denote the scalar Dirichlet Laplacian. For $\sigma>0$,
\begin{equation}\label{eq:domination}
 \bigl|e^{-\sigma\mathcal L_{D,\Omega}}s\bigr|_h
 \leq e^{\sigma\Delta_{\R,\Omega}}|s|_h.
\end{equation}
In particular,
\begin{equation}\label{eq:trace-domination}
 \Tr e^{-\sigma\mathcal L_{D,\Omega}}
 \leq\int_\Omega p_\sigma(x,x)\,dV_g(x).
\end{equation}
\end{lemma}

\begin{proof}
This is the Kato semigroup domination principle \cite{HSU77}. We recall the argument in the present setting. For a smooth solution $s(\sigma)$ of the connection heat equation, compatibility of $D$ with $h$ gives
\[
 (\partial_\sigma-\lap)|s|_h^2=-2|Ds|_h^2.
\]
Set $w_\varepsilon=(|s|_h^2+\varepsilon^2)^{1/2}$. The pointwise Kato inequality implies
\[
 (\partial_\sigma-\lap)w_\varepsilon
 \leq-\frac{\varepsilon^2|Ds|_h^2}
 {(|s|_h^2+\varepsilon^2)^{3/2}}\leq0.
\]
The function $w_\varepsilon-\varepsilon$ has zero boundary values and initial value at most $|s(0)|_h$. The scalar maximum principle, followed by $\varepsilon\downarrow0$ and density, proves \eqref{eq:domination}.

Let $k^D_\sigma(x,y)$ denote the connection heat kernel. Domination gives
$|k^D_\sigma(x,y)|_{\mathrm{op}}\leq p^\Omega_\sigma(x,y)$, where $p^\Omega$ is the scalar Dirichlet kernel. On the diagonal, $k^D_\sigma(x,x)$ is a nonnegative endomorphism of a one-dimensional complex space. Its trace is therefore bounded by $p^\Omega_\sigma(x,x)$. Scalar domain monotonicity gives $p^\Omega_\sigma(x,x)\leq p_\sigma(x,x)$, and integration proves \eqref{eq:trace-domination}.
\end{proof}

\begin{proposition}\label{prop:local-trace}
Fix $b>1$. If $\Omega\subset B(p,bR)$ is smooth and relatively compact, then
\begin{equation}\label{eq:local-trace}
 \Tr e^{-\sigma\mathcal L_{D,\Omega}}
 \leq C(m,b)\left(\frac{R^2}{\sigma}\right)^{m/2},
 \qquad 0<\sigma\leq R^2.
\end{equation}
The constant is independent of $h$, $D$, and $V_p(R)$.
\end{proposition}

\begin{proof}
The on-diagonal heat-kernel estimate gives
\[
 p_\sigma(x,x)\leq
 \frac{C(m)}{\Vol B(x,\sqrt\sigma)}.
\]
For $x\in\Omega$, we have $B(p,R)\subset B(x,(b+1)R)$. Relative volume comparison gives
\[
 \Vol B(x,\sqrt\sigma)
 \geq\left(\frac{\sqrt\sigma}{(b+1)R}\right)^mV_p(R).
\]
Moreover, $\Vol\Omega\leq b^mV_p(R)$. Substitution into \eqref{eq:trace-domination} proves the result. In particular, the possibly small quantity $V_p(R)$ cancels.
\end{proof}

\section{A Dimension Bound for Balanced Sections}

The next estimate is the analytic counting statement used in the main proof. The line bundle and its connection may be chosen separately at each heat-kernel time.

\begin{theorem}\label{thm:balanced-count}
Let $M^m$ be as in Section~2, let $t>0$, and let $(L,h,D)$ be a Hermitian line bundle with a unitary connection. Suppose that smooth sections $s_1,\ldots,s_N$ satisfy
\begin{equation}\label{eq:balanced-sections}
 \sum_{a=1}^N|s_a|_h^2=1,\qquad
 \int_M\ip{s_a}{s_b}_h\,d\nu_t=\frac1N\delta_{ab}.
\end{equation}
If
\begin{equation}\label{eq:energy-def}
 E=\int_M\sum_{a=1}^N|Ds_a|_h^2\,d\nu_t<\infty,
\end{equation}
then
\begin{equation}\label{eq:balanced-count}
 N\leq C(m)(1+tE)^{m/2}.
\end{equation}
\end{theorem}

\begin{proof}
Put $R=\sqrt t$, $V_R=V_p(R)$, and let $a=a(m)$ be as in \eqref{eq:tail}. The sections $u_a=\sqrt N\,s_a$ are orthonormal in $L^2(\nu_t)$. Their exterior mass matrix is
\[
 Q_{ab}=\int_{M\setminus B(p,aR)}\ip{u_a}{u_b}_h\,d\nu_t.
\]
It satisfies
\begin{equation}\label{eq:outside-matrix}
 0\leq Q\leq I_N,\qquad
 \tr Q=N\nu_t(M\setminus B(p,aR))\leq\frac N8.
\end{equation}
At most $N/4$ eigenvalues of $Q$ can exceed $1/2$. Thus there is a coefficient subspace $W\subset\C^N$ of dimension $\ell\geq N/2$ on which $Q\leq I/2$. In particular,
\begin{equation}\label{eq:interior-mass}
 \int_{B(p,aR)}\left|\sum_{a=1}^N c_au_a\right|_h^2d\nu_t
 \geq\frac12|c|^2,\qquad c\in W.
\end{equation}
Choose an orthonormal basis of $W$ and denote the corresponding combinations of the $u_a$ by $v_1,\ldots,v_\ell$.

Take a Lipschitz cutoff $\chi$ with
\[
 0\leq\chi\leq1,\quad \chi=1\text{ on }B(p,aR),\quad
 \supp\chi\subset\overline{B(p,(a+1)R)},\quad
 |d\chi|\leq C/R.
\]
Choose a smooth relatively compact domain $\Omega$ containing $\overline{B(p,(a+1)R)}$ and contained in $B(p,(a+2)R)$. Such a domain exists by completeness and a smooth exhaustion. The sections $\chi v_j$ belong to the Dirichlet form domain on $\Omega$.

Let $G$ and $H$ be their unweighted mass and energy matrices:
\[
 G_{ij}=\int_\Omega\ip{\chi v_i}{\chi v_j}_h\,dV_g,
 \qquad
 H_{ij}=\int_\Omega\ip{D(\chi v_i)}{D(\chi v_j)}_h\,dV_g.
\]
The upper density estimate in \eqref{eq:local-density}, together with \eqref{eq:interior-mass}, gives
\begin{equation}\label{eq:gram-lower}
 G\geq c(m)V_R I_\ell.
\end{equation}
Since the coefficient vectors defining the $v_j$ are orthonormal,
\[
 \sum_{j=1}^\ell|v_j|_h^2\leq\sum_{a=1}^N|u_a|_h^2=N,
 \qquad
 \sum_{j=1}^\ell|Dv_j|_h^2\leq
 N\sum_{a=1}^N|Ds_a|_h^2.
\]
Using the lower density estimate on the support of $\chi$, we obtain
\begin{align}
 \tr H
 &\leq C(m)V_R\int_M\sum_{j=1}^\ell|D(\chi v_j)|_h^2\,d\nu_t\notag\\
 &\leq C(m)V_RN(E+t^{-1}).\label{eq:energy-trace}
\end{align}
Only a gradient bound for the cutoff is needed.

Let $U:\C^\ell\to L^2(\Omega,L;dV_g)$ be the map whose columns are $\chi v_1,\ldots,\chi v_\ell$. The columns of $UG^{-1/2}$ form an unweighted $L^2$-orthonormal family. Its total energy is
\begin{equation}\label{eq:orthonormal-energy}
 \tr(G^{-1}H)\leq C(m)N(E+t^{-1}).
\end{equation}
Let $0\leq\lambda_1\leq\lambda_2\leq\cdots$ be the eigenvalues of $\mathcal L_{D,\Omega}$. The variational principle for sums of eigenvalues gives
\begin{equation}\label{eq:eigen-sum}
 \sum_{j=1}^\ell\lambda_j\leq C(m)N(E+t^{-1}).
\end{equation}
One may verify this form of the variational principle by expanding the orthonormal family in an eigenbasis: the resulting weights lie in $[0,1]$ and sum to $\ell$, so their weighted eigenvalue sum is at least the sum of the first $\ell$ eigenvalues.

For every $\sigma>0$, convexity of the exponential gives
\begin{align}
 \Tr e^{-\sigma\mathcal L_{D,\Omega}}
 &\geq\sum_{j=1}^\ell e^{-\sigma\lambda_j}
 \geq\ell\exp\left(-\frac\sigma\ell\sum_{j=1}^\ell\lambda_j\right)\notag\\
 &\geq\frac N2\exp\bigl(-C(m)\sigma(E+t^{-1})\bigr).
 \label{eq:trace-lower}
\end{align}
For $0<\sigma\leq t$, Proposition~\ref{prop:local-trace} gives the reverse bound
\begin{equation}\label{eq:trace-upper-count}
 \Tr e^{-\sigma\mathcal L_{D,\Omega}}
 \leq C(m)(t/\sigma)^{m/2}.
\end{equation}
Choose $\sigma=t/(1+tE)$. Then $\sigma(E+t^{-1})=1$, and \eqref{eq:trace-lower}--\eqref{eq:trace-upper-count} prove \eqref{eq:balanced-count}.
\end{proof}

\begin{remark}
The proof uses a trace comparison, not a termwise comparison between the scalar and connection eigenvalues. The rank-one hypothesis is also essential to the displayed estimate: a rank-$q$ bundle contributes a factor $q$ to the heat-trace upper bound. Here $N$ counts sections of one line bundle, not summands of a rank-$N$ bundle.
\end{remark}

\section{A Balanced Logarithmic Functional}

Throughout this section, $M^n$ is a complete connected noncompact K\"ahler manifold with $\Ric_g\geq0$. Fix a finite-dimensional subspace $V\subset\OO_d(M)$ which contains $1$, and write $N=\dim_{\C}V$. Choose a basis
\begin{equation}\label{eq:F-basis}
 F=(f_1,\ldots,f_N)^t,\qquad f_1=1.
\end{equation}
In particular, $|F(x)|\geq1$ for all $x$. Define
\begin{equation}\label{eq:positive-space}
 \Pn=\{A\in\operatorname{Herm}(N):A>0,\ \det A=1\},
 \qquad \phi_A=\log(F^*AF),
\end{equation}
and
\begin{equation}\label{eq:functional}
 \Phi_F(A,t)=P_t\phi_A(p),\qquad
 \beta_F(t)=\inf_{A\in\Pn}\Phi_F(A,t).
\end{equation}
For every fixed $A$, these quantities are finite. Indeed, $F^*AF\geq\lambda_{\min}(A)$, while $\phi_A$ has at most logarithmic growth above.

\subsection{Existence of the balancing form}

We record a measure-theoretic balancing statement. Unlike an argument based on integrability of individual logarithms of linear forms, its proof only uses that projective hyperplanes have measure zero.

\begin{lemma}\label{lem:balance}
Let $(X,\nu)$ be a probability space, and let $\xi:X\to\C^N$ be measurable with $|\xi|=1$ almost everywhere. Suppose that
\begin{equation}\label{eq:hyperplane-null}
 \nu\{x:v^*\xi(x)=0\}=0\qquad(0\neq v\in\C^N).
\end{equation}
Then
\[
 J_\nu(A)=\int_X\log(\xi^*A\xi)\,d\nu,\qquad A\in\Pn,
\]
has a unique minimizer $A_\nu$. It satisfies
\begin{equation}\label{eq:balance-general}
 \int_X\frac{A_\nu^{1/2}\xi\xi^*A_\nu^{1/2}}
 {\xi^*A_\nu\xi}\,d\nu=\frac1N I_N.
\end{equation}
If $\nu=\nu_t$ varies continuously in total variation, $\xi$ is fixed, and \eqref{eq:hyperplane-null} holds at every $t$, then $A_{\nu_t}$ is continuous in $t$ and remains in a compact subset of $\Pn$ on every compact time interval.
\end{lemma}

\begin{proof}
The case $N=1$ is immediate. Suppose $N\geq2$. We first claim that for every $\eta>0$ there is $\varepsilon>0$ such that
\begin{equation}\label{eq:uniform-nonconcentration}
 \nu\{|v^*\xi|^2<\varepsilon\}<\eta
 \qquad\text{for every }v\in\C^N\text{ with }|v|=1.
\end{equation}
Otherwise, choose $v_j\to v$ and $\varepsilon_j\downarrow0$ for which the left-hand side is at least $\eta$. Since
\[
 \{|v_j^*\xi|^2<\varepsilon_j\}
 \subset\{|v^*\xi|<\sqrt{\varepsilon_j}+|v_j-v|\},
\]
continuity from above of the measure would give $\nu\{v^*\xi=0\}\geq\eta$, a contradiction.

Let $e^L$ be the largest eigenvalue of $A\in\Pn$. Then $L\geq0$ and the smallest eigenvalue is at least $e^{-(N-1)L}$. Choose a unit eigenvector $v$ corresponding to $e^L$. On the complement of the set in \eqref{eq:uniform-nonconcentration},
\[
 \log(\xi^*A\xi)\geq L+\log\varepsilon;
\]
everywhere it is at least $-(N-1)L$. For $L$ sufficiently large, these bounds give
\begin{equation}\label{eq:coercivity}
 J_\nu(A)\geq(1-N\eta)L+(1-\eta)\log\varepsilon.
\end{equation}
Take $\eta=1/(2N)$. It follows that $J_\nu$ is coercive on $\Pn$. A bound for the largest eigenvalue, together with $\det A=1$, also bounds the inverse. Thus its sublevel sets are compact. The integrand is bounded uniformly in $x$ when $A$ ranges in a compact subset of $\Pn$, so $J_\nu$ is continuous and attains its minimum.

To prove uniqueness, consider a nonconstant geodesic
\[
 A(s)=A_0^{1/2}e^{sB}A_0^{1/2},\qquad B=B^*,\quad\tr B=0.
\]
After diagonalizing $B$, the second derivative of the integrand has the form
\begin{equation}\label{eq:variance}
 \frac{d^2}{ds^2}\log\left(\sum_je^{s\lambda_j}|w_j|^2\right)
 =\sum_j\theta_j\lambda_j^2-
 \left(\sum_j\theta_j\lambda_j\right)^2,\qquad
 \theta_j=\frac{e^{s\lambda_j}|w_j|^2}
 {\sum_ke^{s\lambda_k}|w_k|^2}.
\end{equation}
It is nonnegative. It vanishes only when $w$ belongs to one eigenspace of $B$. Since $B$ is nonscalar, those eigenspaces are proper complex subspaces. Their inverse images under $A_0^{1/2}$ are also proper, and each is contained in a hyperplane. By \eqref{eq:hyperplane-null}, their union has measure zero. Hence $J_\nu$ is strictly convex along every nonconstant geodesic. This proves uniqueness.

At the minimum, differentiate along $A_\nu^{1/2}e^{sB}A_\nu^{1/2}$ for every traceless Hermitian $B$. The result is that the matrix on the left of \eqref{eq:balance-general} is orthogonal to all such $B$. Its trace is one, so it equals $I_N/N$.

Finally, on a compact time interval, \eqref{eq:uniform-nonconcentration} holds with a common $\varepsilon$. If not, the same argument with a subsequence $t_j\to t$ and total-variation convergence would contradict \eqref{eq:hyperplane-null} at $t$. Thus \eqref{eq:coercivity} is uniform there. Since $J_{\nu_t}(I)=0$, all minimizers lie in one compact subset of $\Pn$. Joint continuity and uniqueness show that every convergent subsequence of minimizers tends to the minimizer at the limiting time. This proves continuity.
\end{proof}

The balancing construction is classical in projective geometry; compare \cite{BLY94,Donaldson09}. The direct proof above records the particular nonconcentration and parameter properties that we need. None of its coercivity constants is claimed to be uniform over all choices of $V$.

Apply Lemma~\ref{lem:balance} to
\[
 \xi(x)=F(x)/|F(x)|,\qquad d\nu=d\nu_t.
\]
For every $v\neq0$, the function $v^*F$ is a nonzero holomorphic function, so its zero set has Riemannian volume zero. The heat measure has a smooth positive density. Thus \eqref{eq:hyperplane-null} holds. The difference between $\Phi_F(A,t)$ and $J_{\nu_t}(A)$ is the quantity $P_t\log|F|^2(p)$, independent of $A$. Consequently, the infimum in \eqref{eq:functional} is attained at a unique matrix $A_t$, continuous for $t>0$.

Set
\begin{equation}\label{eq:balanced-data}
 g^{(t)}=A_t^{1/2}F,\qquad
 S_t=|g^{(t)}|^2=F^*A_tF,\qquad
 \phi_t=\log S_t.
\end{equation}
Then
\begin{equation}\label{eq:balance-heat}
 \int_M\frac{g^{(t)}(g^{(t)})^*}{S_t}\,d\nu_t
 =\frac1N I_N.
\end{equation}

\subsection{The envelope differentiation formula}

\begin{proposition}\label{prop:envelope}
The function $\beta_F$ is continuously differentiable on $(0,\infty)$, and
\begin{equation}\label{eq:envelope}
 \beta_F'(t)=\int_M\lap\phi_t\,d\nu_t\geq0.
\end{equation}
No derivative of the map $t\mapsto A_t$ is required.
\end{proposition}

\begin{proof}
We first check the hypotheses for differentiating at a fixed matrix. The real and imaginary parts of each $f_a$ are harmonic. The Cheng--Yau local gradient estimate \cite{ChengYau75}, applied on unit balls, gives
\[
 |df_a|(x)\leq C_{a,d}(1+r(x))^d.
\]
If $A$ belongs to a compact subset $K\subset\Pn$, then
\begin{equation}\label{eq:matrix-local-bounds}
 c_K|F|^2\leq F^*AF\leq C_K|F|^2,
 \qquad F^*AF\geq c_K>0.
\end{equation}
Writing $g=A^{1/2}F$ and $S=|g|^2$, holomorphicity gives
\begin{equation}\label{eq:log-hessian}
 (\log S)_{i\bar j}
 =\frac{\sum_a(Z_i g_a)\overline{Z_jg_a}}{S}
 -\frac{\bigl(\sum_a\overline{g_a}Z_i g_a\bigr)
 \bigl(\sum_b g_b\overline{Z_jg_b}\bigr)}{S^2}.
\end{equation}
This Hermitian matrix is nonnegative. Equations~\eqref{eq:matrix-local-bounds}--\eqref{eq:log-hessian} and the gradient bounds imply
\begin{equation}\label{eq:phi-growth}
 |\phi_A|+|d\phi_A|+|\lap\phi_A|
 \leq C_{K,V,d}(1+r)^{q_{V,d}}
\end{equation}
for some finite exponent, uniformly over $A\in K$. Only first derivatives of the holomorphic functions occur in \eqref{eq:log-hessian}.

Lemma~\ref{lem:heat-derivative} now gives
\begin{equation}\label{eq:partial-Phi}
 \partial_t\Phi_F(A,t)=P_t(\lap\phi_A)(p),
\end{equation}
and this derivative is jointly continuous on compact matrix sets and compact positive time intervals. In particular, no global curvature upper bound is needed for this step.

Fix $t>0$. For $h>0$ sufficiently small, minimality at $t$ and $t+h$ gives
\begin{align*}
 \Phi_F(A_{t+h},t+h)-\Phi_F(A_{t+h},t)
 &\leq\beta_F(t+h)-\beta_F(t)\\
 &\leq\Phi_F(A_t,t+h)-\Phi_F(A_t,t).
\end{align*}
Divide by $h$ and use $A_{t+h}\to A_t$ and the joint continuity in \eqref{eq:partial-Phi}. Both bounds tend to $\partial_t\Phi_F(A_t,t)$. The argument for $h<0$ is identical after reversing the divided inequalities. This proves \eqref{eq:envelope}, and continuity of the derivative follows from the same facts. Nonnegativity follows from \eqref{eq:log-hessian}.
\end{proof}

\begin{remark}\label{rem:basis}
The derivative in \eqref{eq:envelope} is intrinsic to $V$ and $p$. If $\widetilde F=BF$ for $B\in\GL(N,\C)$, then the change of variable
\[
 A\longmapsto |\det B|^{-2/N}B^*AB
\]
in \eqref{eq:functional} shows that
\begin{equation}\label{eq:basis-change}
 \beta_{\widetilde F}(t)=\beta_F(t)+\frac2N\log|\det B|.
\end{equation}
Thus $\beta_F'(t)$ and $\beta_F(t)-\beta_F(1)$ do not depend on the chosen basis. The condition numbers of $A_t$ may nevertheless diverge as $t\to0$ or $t\to\infty$.
\end{remark}

\section{Proof of the Main Theorem}

\subsection{The Chern-connection energy identity}

We isolate the use of holomorphicity. The formula is local, but the resulting line bundle and sections will be used globally.

\begin{lemma}\label{lem:chern}
Let $g_1,\ldots,g_N$ be holomorphic functions on a K\"ahler manifold, with no common zero. Put
\[
 S=\sum_{a=1}^N|g_a|^2,\qquad \phi=\log S,
 \qquad h=S^{-1}
\]
on the trivial holomorphic line bundle $L=M\times\C$. Regard each $g_a$ as the section $g_ae$ in the global holomorphic frame $e$ with $|e|_h^2=h$. Then the Chern connection satisfies
\begin{equation}\label{eq:chern-energy}
 \sum_{a=1}^N|g_a|_h^2=1,\qquad
 \sum_{a=1}^N|D^h g_a|_h^2
 =\lapc\phi=\frac12\lap\phi.
\end{equation}
\end{lemma}

\begin{proof}
The Chern connection in the frame $e$ is
\begin{equation}\label{eq:chern-connection}
 D^h(g_ae)=(\partial g_a-g_a\partial\phi)e,
 \qquad (D^h)^{0,1}(g_ae)=0.
\end{equation}
These are the usual Chern-connection conventions; see \cite[Chapter~V]{Demailly12}. At a point, choose a unitary frame $Z_1,\ldots,Z_n$. Since the $(0,1)$ covariant derivatives vanish, the full real derivative norm equals the sum of the $(1,0)$ derivative norms. Therefore
\begin{align}
 \sum_a|D^h g_a|_h^2
 &=\sum_i\left(
 \frac{\sum_a|Z_i g_a|^2}{S}
 -\frac{|\sum_a\overline{g_a}Z_i g_a|^2}{S^2}
 \right)\notag\\
 &=\sum_i\phi_{i\bar i}
 =\lapc\phi.\label{eq:chern-calculation}
\end{align}
The last equality is \eqref{eq:log-hessian}. The remaining statements follow from the definitions and \eqref{eq:lap-convention}.
\end{proof}

\begin{remark}
Replacing $D^h$ by the ordinary differential is not harmless. Indeed, if $u_a=g_a/\sqrt S$, a direct calculation with the same real norms gives
\begin{equation}\label{eq:ordinary-energy}
 \sum_a|du_a|^2
 =\frac12\lap\phi+\frac14|d\phi|^2.
\end{equation}
The Chern connection removes the last term. The energy used in Theorem~\ref{thm:balanced-count} is the connection energy in \eqref{eq:chern-energy}, not the left-hand side of \eqref{eq:ordinary-energy}.
\end{remark}

\begin{proposition}\label{prop:slope}
For $V$, $F$, and $\beta_F$ as in Section~4, there is a constant $C(n)$, independent of $V$, such that
\begin{equation}\label{eq:slope}
 \dim_{\C}V\leq C(n)\bigl(1+t\beta_F'(t)\bigr)^n
 \qquad\text{for every }t>0.
\end{equation}
\end{proposition}

\begin{proof}
Fix $t>0$. Give the trivial line bundle the metric $h_t=S_t^{-1}$ from \eqref{eq:balanced-data}, and keep this metric fixed in the application of the connection heat-trace estimate. The sections $g_a^{(t)}$ satisfy
\[
 \sum_a|g_a^{(t)}|_{h_t}^2=1,\qquad
 \int_M\ip{g_a^{(t)}}{g_b^{(t)}}_{h_t}\,d\nu_t
 =\frac1N\delta_{ab},
\]
by \eqref{eq:balance-heat}. Equations~\eqref{eq:envelope} and \eqref{eq:chern-energy} give
\begin{equation}\label{eq:E-beta}
 E_t:=\int_M\sum_a|D^{h_t}g_a^{(t)}|_{h_t}^2\,d\nu_t
 =\frac12\beta_F'(t).
\end{equation}
This energy is finite by \eqref{eq:phi-growth}. Apply Theorem~\ref{thm:balanced-count} in real dimension $m=2n$ and absorb the fixed factor $1/2$ into the dimensional constant. This proves \eqref{eq:slope}.
\end{proof}

\subsection{Growth of the minimized heat average}

\begin{proof}[Proof of Theorem~\ref{thm:main}]
It suffices to bound every finite-dimensional subspace of $\OO_d(M)$ by the same constant. Enlarge such a subspace by the constants and denote the resulting space by $V$. Let $F$ and $N$ be as in \eqref{eq:F-basis}. There is $C_F>0$ such that
\[
 |F(x)|^2\leq C_F(1+r(x))^{2d}.
\]
Taking $A=I_N$ in \eqref{eq:functional} gives
\begin{equation}\label{eq:upper-start}
 \beta_F(t)\leq\log C_F+2d\int_M\log(1+r)\,d\nu_t.
\end{equation}
Since $(1+r)^2\leq2(1+r^2)$, Jensen's inequality and \eqref{eq:moment} imply
\begin{align}
 \int_M\log(1+r)\,d\nu_t
 &\leq\frac12\log\left(2\int_M(1+r^2)\,d\nu_t\right)\notag\\
 &\leq\frac12\log\bigl(2+2C(n)t\bigr).
 \label{eq:log-moment}
\end{align}
Consequently,
\begin{equation}\label{eq:beta-upper}
 \beta_F(t)\leq d\log t+C_{F,d,n},\qquad t\geq1.
\end{equation}
The additive constant in this bound may depend on the basis and its growth constants. Its coefficient of $\log t$ does not.

Let $C_0(n)\geq1$ be the constant in \eqref{eq:slope}. That estimate yields
\[
 t\beta_F'(t)\geq (N/C_0(n))^{1/n}-1.
\]
Integrating from $1$ to $t\geq1$, we find
\begin{equation}\label{eq:beta-lower}
 \beta_F(t)-\beta_F(1)
 \geq\bigl((N/C_0(n))^{1/n}-1\bigr)\log t.
\end{equation}
Compare \eqref{eq:beta-upper} and \eqref{eq:beta-lower}, divide by $\log t$, and let $t\to\infty$. We obtain
\[
 (N/C_0(n))^{1/n}-1\leq d,
\]
which is exactly $N\leq C_0(n)(1+d)^n$.

This estimate holds for every finite-dimensional subspace of $\OO_d(M)$. If the whole space had larger dimension, it would contain a finite-dimensional subspace violating the estimate. Thus \eqref{eq:main} follows without a prior finite-dimensionality assumption. Equation~\eqref{eq:main-large} follows from $1+d\leq2d$ when $d\geq1$.
\end{proof}

\begin{remark}
The proof does not require control of $A_t$ as $t\to\infty$. Compactness of the minimizing matrices is used only locally in time, to justify the envelope formula. In contrast, the constant in \eqref{eq:slope} is independent of those matrices, since the connection heat-trace bound is independent of the line-bundle metric and curvature.
\end{remark}

\begin{remark}
If $0\leq d<1$, then $\OO_d(M)=\C$. Indeed, the local gradient estimate applied on balls of radius tending to infinity gives $|df|(x)\leq C_fR^{d-1}$ at each fixed $x$, and hence $df=0$. This is consistent with the normalization in Theorem~\ref{thm:main}.
\end{remark}

\section{Examples and Scope}

\subsection{An integrated kernel estimate for global growth spaces}

For $S>0$, the inner product
\begin{equation}\label{eq:ball-inner}
 \ip{f}{g}_{S}=\int_{B(p,S)}\overline f g\,dV_g
\end{equation}
is positive definite on $\OO_d(M)$ by the identity theorem. Let $f_1,\ldots,f_{N_d}$ be an orthonormal basis for this inner product and define
\begin{equation}\label{eq:global-kernel}
 K_{d,S}(x,x)=\sum_{a=1}^{N_d}|f_a(x)|^2,
 \qquad N_d=\dim_{\C}\OO_d(M).
\end{equation}
This definition is independent of the orthonormal basis.

\begin{corollary}\label{cor:kernel}
Under the assumptions of Theorem~\ref{thm:main}, for every $d\geq0$ and $0<R\leq S$,
\begin{equation}\label{eq:global-trace}
 \int_{B(p,R)}K_{d,S}(x,x)\,dV_g(x)
 \leq C(n)(1+d)^n.
\end{equation}
In particular, for $0<\delta\leq1/2$,
\begin{equation}\label{eq:global-BT}
 \sup_{R>0}\int_{B(p,R)}
 K_{\delta^{-1},(1+\delta)R}(x,x)\,dV_g(x)
 \leq C(n)\delta^{-n}.
\end{equation}
\end{corollary}

\begin{proof}
Orthonormality gives
\[
 \int_{B(p,S)}K_{d,S}(x,x)\,dV_g(x)=N_d.
\]
Restrict the integral to $B(p,R)$ and apply Theorem~\ref{thm:main}. For \eqref{eq:global-BT}, use $1+\delta^{-1}\leq2\delta^{-1}$.
\end{proof}

This is a growth-cutoff estimate for global functions. It is not an estimate for the full local Bergman space
\[
 A^2(B(p,S))=\{f\in\OO(B(p,S)):\int_{B(p,S)}|f|^2\,dV_g<\infty\}.
\]
In fact, its kernel satisfies $K_{d,S}\leq K_{A^2(B(p,S))}$ on the ball, which is the wrong direction for obtaining an upper bound for the local kernel. No passage $d\to\infty$ in \eqref{eq:global-trace} proves such an upper bound.

\subsection{Algebraic dependence}

Let
\[
 \OO_{\mathrm{poly}}(M)=\bigcup_{d\geq0}\OO_d(M).
\]
It is an integral domain, since $M$ is connected.

\begin{corollary}\label{cor:algebraic}
Under the assumptions of Theorem~\ref{thm:main},
\begin{equation}\label{eq:transdeg}
 \trdeg_{\C}\Frac\OO_{\mathrm{poly}}(M)\leq n.
\end{equation}
\end{corollary}

\begin{proof}
Suppose that $u_1,\ldots,u_k$ are algebraically independent elements of the fraction field. Write $u_j=a_j/b_j$, where $a_j,b_j\in\OO_D(M)$ for a common $D\geq1$, and no $b_j$ is identically zero. For each positive integer $l$, the $(l+1)^k$ functions
\[
 \prod_{j=1}^k a_j^{\alpha_j}b_j^{l-\alpha_j},
 \qquad 0\leq\alpha_j\leq l,
\]
belong to $\OO_{kDl}(M)$ and are linearly independent. Indeed, division by $\prod_jb_j^l$ in the fraction field turns a linear relation into a polynomial relation among the $u_j$. Theorem~\ref{thm:main} therefore gives
\[
 (l+1)^k\leq C(n)(1+kDl)^n.
\]
Letting $l\to\infty$ proves $k\leq n$.
\end{proof}

The same algebraic counting argument was used under stronger curvature assumptions in \cite[Corollary~5]{Liu16}. Corollary~\ref{cor:algebraic} is only a transcendence-degree statement. It does not assert finite generation of $\OO_{\mathrm{poly}}(M)$.

\subsection{Two model calculations}

\begin{example}[Euclidean homogeneous spaces]\label{ex:euclidean}
Let $M=\C^n$ with its Euclidean metric and base point $0$. Let $V$ be the span of all monomials $z^\alpha$ with $|\alpha|\leq D$, where $D$ is a nonnegative integer, and use these monomials as the basis $F$. Then
\begin{equation}\label{eq:euclidean-beta}
 \beta_F(t)=\beta_F(1)+\frac{nD}{n+1}\log t.
\end{equation}
\end{example}

\begin{proof}
The measure $\nu_t$ is the image of $\nu_1$ under $z\mapsto\sqrt t\,z$. If $B_t$ is the diagonal matrix with entries $t^{|\alpha|/2}$, then $F(\sqrt t\,z)=B_tF(z)$. Put
\[
 N=\binom{n+D}{n},\qquad q_D=\sum_{|\alpha|\leq D}|\alpha|.
\]
For $A\in\Pn$, the matrix $t^{-q_D/N}B_tAB_t$ has determinant one. Changing variables in the heat average and then in the minimizing matrix shows
\[
 \beta_F(t)=\beta_F(1)+(q_D/N)\log t.
\]
The elementary identities
\[
 q_D=\sum_{j=0}^Dj\binom{n+j-1}{n-1}
 =n\binom{n+D}{n+1},\qquad
 q_D/N=\frac{nD}{n+1}
\]
prove \eqref{eq:euclidean-beta}. Thus the slope of the balanced functional records the average homogeneous degree in this example.
\end{proof}

\begin{example}[Zero asymptotic volume ratio]\label{ex:collapse}
Let $1\leq k<n$, and let $Y^{n-k}$ be a compact connected K\"ahler manifold with nonnegative Ricci curvature. Endow $M=\C^k\times Y$ with the product metric. Every holomorphic function on $M$ is constant on each compact fiber, so
\begin{equation}\label{eq:product-dimension}
 \OO_d(M)\simeq\OO_d(\C^k),\qquad
 \dim_{\C}\OO_d(M)=\binom{k+\lfloor d\rfloor}{k}.
\end{equation}
The product distance and compactness of $Y$ show that $V_p(R)\asymp R^{2k}$ for large $R$. In particular, $R^{-2n}V_p(R)\to0$. Theorem~\ref{thm:main} therefore includes manifolds with zero asymptotic volume ratio.
\end{example}

\subsection{The scope of the estimate}
\label{subsec:scope}
The range of the growth order in \cite[Question~1]{Liu21}
requires a clarification. With the definition
\eqref{eq:holomorphic-space}, the space $\OO_d(M)$ contains the
constant functions for every $d\geq0$. Consequently, an estimate
\[
 \dim_{\C}\OO_d(M)\leq C(n)d^n
\]
cannot hold with the same finite constant $C(n)$ for all real
$d>0$: its left-hand side is at least one, whereas its right-hand
side tends to zero as $d\downarrow0$. This obstruction already
occurs on $\C^n$. It concerns the constant functions at small
growth orders, rather than the power of the estimate as
$d\to\infty$.

Under the assumptions of Theorem~\ref{thm:main}, the small-growth
range can in fact be described exactly. Combining this fact with
the main theorem gives the following formulation.

\begin{corollary}\label{cor:normalized-question}
Let $M^n$ be a complete connected noncompact K\"ahler manifold
with $\Ric_g\geq0$. Then
\begin{equation}\label{eq:scope-small-growth}
 \OO_d(M)=\C,\qquad 0\leq d<1.
\end{equation}
Moreover, there is a constant $C(n)>0$ such that
\begin{equation}\label{eq:scope-quotient}
 \dim_{\C}\bigl(\OO_d(M)/\C\bigr)
 =\dim_{\C}\OO_d(M)-1
 \leq C(n)d^n,\qquad d>0.
\end{equation}
Equivalently, if $p\in M$ is fixed and
\begin{equation}\label{eq:scope-based-space}
 \OO_d^0(M,p)=\{f\in\OO_d(M):f(p)=0\},
\end{equation}
then
\begin{equation}\label{eq:scope-based-bound}
 \dim_{\C}\OO_d^0(M,p)\leq C(n)d^n,\qquad d>0.
\end{equation}
In particular, after changing the dimensional constant if
necessary,
\begin{equation}\label{eq:scope-all-orders}
 \dim_{\C}\OO_d(M)\leq 1+C(n)d^n,\qquad d\geq0.
\end{equation}
\end{corollary}

\begin{proof}
Let $0\leq d<1$ and $f\in\OO_d(M)$. The real and imaginary
parts of $f$ are harmonic. The local gradient estimate of
Cheng--Yau \cite{ChengYau75}, applied to these harmonic functions
after adding suitable constants on each ball, gives
\[
 |df|(q)
 \leq \frac{C(n)}{R}\sup_{B(q,2R)}|f|
 \leq \frac{C(n)C_f}{R}
       \bigl(1+d_g(p,q)+2R\bigr)^d
\]
for every fixed $q\in M$ and every $R>0$. Letting $R\to\infty$
yields $df(q)=0$. Since $q$ is arbitrary and $M$ is connected,
$f$ is constant. This proves \eqref{eq:scope-small-growth}.

For $d\geq1$, Theorem~\ref{thm:main} gives
\[
 \dim_{\C}\OO_d(M)-1
 \leq \dim_{\C}\OO_d(M)
 \leq C(n)(1+d)^n
 \leq 2^nC(n)d^n.
\]
For $0<d<1$, the left-hand side is zero by
\eqref{eq:scope-small-growth}. These two cases prove
\eqref{eq:scope-quotient} with one dimensional constant.

Evaluation at $p$ gives the direct sum
\[
 \OO_d(M)=\C\oplus\OO_d^0(M,p),
 \qquad
 f=f(p)+(f-f(p)).
\]
Thus \eqref{eq:scope-based-bound} is equivalent to
\eqref{eq:scope-quotient}. Adding back the constant functions
proves \eqref{eq:scope-all-orders} for $d>0$; the case $d=0$
follows from \eqref{eq:scope-small-growth}.
\end{proof}

Accordingly, Theorem~\ref{thm:main} gives an affirmative answer
to the large-growth-order formulation of
\cite[Question~1]{Liu21}, namely
\[
 \dim_{\C}\OO_d(M)\leq C(n)d^n,\qquad d\geq1,
\]
without assuming Euclidean volume growth. For all real $d>0$,
the corresponding statement is
\eqref{eq:scope-quotient}, or equivalently
\eqref{eq:scope-all-orders}. The formula in Question~1, read
literally with a constant independent of $d$ on the whole
interval $(0,\infty)$, must therefore be adjusted by retaining
the contribution of the constant functions.

The power remains sharp after this adjustment. Indeed,
\eqref{eq:euclidean-dimension} gives
\[
 \dim_{\C}\bigl(\OO_d(\C^n)/\C\bigr)
 =\binom{n+\lfloor d\rfloor}{n}-1
 =\frac{d^n}{n!}+O_n(d^{n-1})
 \qquad\text{as }d\to\infty.
\]
Hence no uniform estimate with power $n-\varepsilon$,
$\varepsilon>0$, can replace \eqref{eq:scope-quotient}.
The present argument does not determine the optimal
dimensional constant or the optimal asymptotic leading
coefficient. It also does not establish the precise Euclidean
dimension comparison or a rigidity statement for an equality
case.


\end{document}